\pdfoutput=1
\documentclass[10pt,a4paper]{article}

\usepackage[numbers,sort&compress]{natbib}
\usepackage[english]{babel}

\usepackage[a4paper,top=2cm,bottom=2cm,left=2.6cm,right=2.8cm,marginparwidth=1.75cm]{geometry}

\usepackage[
  activate={true,nocompatibility},
  final,
  tracking=true,
  kerning=true,
  factor=1100,
  stretch=10,
  shrink=10
]{microtype}
\microtypecontext{spacing=nonfrench}

\usepackage[utf8]{inputenc}
\usepackage[T1]{fontenc}

\usepackage{amsmath,amssymb,amsfonts,amsthm,mathtools}
\usepackage{bm}
\usepackage{bbm}
\usepackage{esvect}
\allowdisplaybreaks

\usepackage{graphicx}
\usepackage{booktabs}
\usepackage{adjustbox}
\usepackage{caption}
\usepackage{subcaption}
\usepackage{float}
\usepackage{placeins}

\usepackage{tikz}
\usetikzlibrary{decorations.pathreplacing,decorations.markings}
\usepackage{pifont}

\usepackage{enumitem}
\usepackage{verbatim}

\usepackage{authblk}

\usepackage{xcolor}
\definecolor{mycitecolor}{RGB}{70,10,200}
\usepackage[colorlinks=true,linkcolor=blue,citecolor=mycitecolor,urlcolor=black]{hyperref}
\usepackage[nameinlink,noabbrev]{cleveref}

\newtheoremstyle{myplain}
  {3pt}   
  {3pt}   
  {\itshape}
  {}
  {\bfseries}
  {.}
  { }
  {}

\theoremstyle{myplain}
\newtheorem{theorem}{Theorem}[section]
\newtheorem{lemma}[theorem]{Lemma}

\newtheorem{corollary}[theorem]{Corollary}
\newtheorem{conjecture}[theorem]{Conjecture}

\newtheorem{observation}[theorem]{Observation}
\newtheorem{claim}[theorem]{Claim}

\theoremstyle{definition}

\theoremstyle{remark}

\newcommand{\EE}{\mathbb{E}}
\newcommand{\bbone}{\mathbf{1}}

\title{\bfseries On the Borodin--Kostochka conjecture for graphs with large maximum degree}

\author{\bfseries Feng Liu\footnote{Email: liufeng0609@126.com.}}
\author{\bfseries Shuang Sun\footnote{Email: chocolatesun@sjtu.edu.cn.}}
\author{\bfseries Yan Wang\footnote{Email: yan.w@sjtu.edu.cn (corresponding author).}}
\author{\bfseries Jiasheng Zeng\footnote{Email: jasonzeng@mail.ustc.edu.cn.}}

\affil{\footnotesize School of Mathematical Sciences, Shanghai Jiao Tong University, 800 Dongchuan Road, Shanghai 200240, China}

\date{}
\begin{document}
\maketitle

\begin{abstract}
The Borodin--Kostochka conjecture states that every graph $G$ with maximum degree $\Delta(G)\ge 9$ satisfies $\chi(G)\le \max\{\omega(G),\Delta(G)-1\}$. In this paper, we verify this conjecture for graphs with sufficiently large maximum degree. More precisely, we prove that every graph $G$ with maximum degree $\Delta \ge 5.3\times 10^6$ and clique number $\omega(G)<\Delta$ satisfies $\chi(G)\le \Delta-1$. This improves a longstanding result of Reed.

\smallskip
\noindent \textbf{Keywords:} Graph coloring; Borodin--Kostochka conjecture; maximum degree.

\noindent \textbf{AMS Subject Classification:} 05C15, 05C85.

\end{abstract}

\section{Introduction}
All graphs considered in this paper are finite, simple, and connected. For a graph $G$, we denote by $V(G)$ and $E(G)$ its vertex set and edge set, respectively, and by $\Delta(G)$ its maximum degree. For a set $S\subseteq V(G)$, we use $G[S]$ to denote the subgraph of $G$ induced by $S$, and write $G-S$ for $G[V(G)\setminus S]$. For a vertex $x\in V(G)$ and a subgraph $H$ of $G$, let $N_H(x)$ denote the set of neighbors of $x$ in $V(H)$. Let $X,Y$ be disjoint subsets of $V(G)$. For a vertex $v\in V(G)\setminus X$, we say that $v$ is {\em complete} to $X$ if $v$ is adjacent to every vertex in $X$, and {\em anticomplete} to $X$ if $v$ is nonadjacent to every vertex in $X$. We say that $X$ is {\em complete} (resp., {\em anticomplete}) to $Y$ if every vertex in $X$ is complete (resp., anticomplete) to $Y$. A subset $X$ of $V(G)$ is a {\em clique} if $G[X]$ is complete. The {\em clique number} of $G$, denoted by $\omega(G)$, is the maximum size of a clique in $G$.

A $k$-{\em coloring} of $G$ is a mapping $\varphi:V(G)\to \{1,2,\ldots,k\}$ such that $\varphi(u)\neq \varphi(v)$ whenever $u$ and $v$ are adjacent in $G$. The {\em chromatic number} of $G$ is the minimum integer $k$ such that $G$ admits a $k$-coloring. By greedily coloring a graph $G$, we have $\chi(G)\le \Delta(G)+1$. A classical theorem of Brooks~\cite{Brooks1941} shows that this bound can be improved as follows.

\begin{theorem}[Brooks~\cite{Brooks1941}]\label{thm:brooks}
If $G$ is a graph with $\Delta(G)\ge 3$, then $\chi(G)\le \max\{\Delta(G),\omega(G)\}$.
\end{theorem}

It is therefore natural to ask whether one can replace $\Delta(G)$ in Brooks' theorem by $\Delta(G)-1$ when $\omega(G)<\Delta(G)$. In 1977, Borodin and Kostochka~\cite{Borodin-Kostochka1977} proposed the following conjecture.

\begin{conjecture}[Borodin-Kostochka~\cite{Borodin-Kostochka1977}]\label{conj:BK}
If $G$ is a graph with $\Delta(G)\ge 9$, then $\chi(G)\le \max\{\Delta(G)-1,\omega(G)\}$.
\end{conjecture}

By Theorem~\ref{thm:brooks}, Conjecture~\ref{conj:BK} is trivial when $\omega(G)\ge \Delta(G)$. Thus it suffices to show that every graph $G$ with $\Delta(G)\ge 9$ and $\omega(G)\le \Delta(G)-1$ satisfies $\chi(G)\le \Delta(G)-1$. Reed~\cite{Reed1999} proved this for all graphs with $\Delta(G)\ge 10^{14}$. Conjecture~\ref{conj:BK} has also been verified for a number of hereditary graph classes defined by forbidding induced subgraphs; see, for example, \cite{Cranston-Rabern2013,Cranston-Lafayette-Rabern2022,Gupta-Pradhan2021,Lan-Liu-Zhou2024,Chen-Lan-Zhou2024,Wu-Wu2025,Chen-Lan-Lin-Zhou2024}. Most of these proofs reduce to the case when $\Delta=9$ and then analyze a minimal critical counterexample, following the approach of Catlin~\cite{Catlin1976} and Kostochka~\cite{Kostochka1980}.

In this paper, we study the colorings of general graphs with large maximum degree, in the spirit of Reed's asymptotic result. Our goal is to prove the existence of a $(\Delta-1)$-coloring under an improved degree assumption with explicit constants. The main result of this paper is the following.

\begin{theorem}\label{main}
Every graph $G$ with $\Delta(G)\ge 5.3\times 10^6$ satisfies
$\chi(G)\le \max\{\Delta(G)-1,\omega(G)\}$.
\end{theorem}

We first derive a structural decomposition of a minimal counterexample, isolating large cliques and near-cliques. We then use the Lov\'asz Local Lemma to obtain a partial coloring satisfying the hypotheses of a deterministic extension lemma.
The rest of the paper is organized as follows. In Section~\ref{decomposition}, we state the structural properties of a minimal counterexample and derive a partition that will be used in the later coloring argument.
Finally, in Section~\ref{proper-coloring}, we establish the probabilistic coloring step needed to complete the proof of Theorem~\ref{main}.

\section{Structure of a minimal counterexample to Theorem~\ref{main}}\label{decomposition}
In this section, we investigate the structure of a minimal counterexample to Theorem~\ref{main}.
A graph $G$ is called a \emph{minimal counterexample to Theorem~\ref{main} with maximum degree $\Delta$} if $\Delta(G)=\Delta$, $\omega(G)\le \Delta-1$, and $\chi(G)\ge \Delta$, and every graph $H$ with $\Delta(H)=\Delta, |V(H)|<|V(G)|$ and $\omega(H)\le \Delta-1$ satisfies $\chi(H)\le \Delta-1$. The purpose of this section is to derive a global partition that will later be used in the probabilistic coloring argument. We begin with a structural lemma due to Reed concerning the $(\Delta-1)$-cliques in a minimal counterexample.

\begin{lemma}[Reed \cite{Reed1999}]\label{p=1}
If $G$ is a minimal counterexample to Theorem~\ref{main} of maximum degree $\Delta$, and $K$ is a $(\Delta-1)$-clique of $G$, then no vertex of $G-K$ is adjacent to more than four vertices of $K$. Furthermore, at most four vertices of $K$ have degree $\Delta-1$.
\end{lemma}

Besides these structural restrictions on $(\Delta-1)$-cliques, we shall also use the following simple consequence of the minimality of $G$ throughout this section.

\begin{observation}\label{obs:critical}
Let $G$ be a minimal counterexample to Theorem~\ref{main} with maximum degree $\Delta$. Then $G$ is $\Delta$-critical, and hence $\delta(G)\ge \Delta-1$.
\end{observation}
\begin{proof}
Since $G$ is a minimal counterexample, Theorem~\ref{thm:brooks} implies that $\omega(G)\le \Delta-1$ and $\chi(G)=\Delta$. It suffices to show that $\chi(G-v)\le \Delta-1$ for every $v\in V(G)$. Suppose to the contrary that there exists a vertex $v\in V(G)$ such that $\chi(G-v)=\Delta$. Note that $\omega(G-v)\le \omega(G)\le \Delta-1$. If $\Delta(G-v)\le \Delta-1$, then this contradicts Theorem~\ref{thm:brooks}. Therefore, $\Delta(G-v)=\Delta$. But then $G-v$ is a smaller graph with maximum degree $\Delta$, clique number at most $\Delta-1$, and chromatic number $\Delta$, contradicting the minimality of $G$.
\end{proof}

We shall use the following consequence repeatedly: every proper induced subgraph of $G$ is $(\Delta-1)$-colorable. Indeed, if such a subgraph has maximum degree at most $\Delta-1$, this follows from Brooks' theorem; otherwise it has maximum degree $\Delta$ and is smaller than $G$, so it is $(\Delta-1)$-colorable by the minimality of $G$.

We next establish a dense global structure for a minimal counterexample. The following lemma is the key structural input in the paper: it shows that every sufficiently dense induced subgraph of bounded size is forced to be either a clique or a clique together with one extra vertex. Later, this will allow us to partition the whole graph into large clique-like pieces and a sparse remainder.
\begin{lemma}\label{Hcupv}
Let $G$ be a minimal counterexample to Theorem~\ref{main}, and let $H$ be an induced subgraph of $G$ with at most $\Delta+c$ vertices, where $6\le c\le \Delta/10$. Suppose that every vertex of $H$ has at least $4\Delta/5$ neighbors in $H$. Then $H$ is either a clique, or consists of a clique $C_H$ with fewer than $\Delta-1$ vertices together with a vertex $v_H$.
\end{lemma}
\begin{proof}[\bf Proof]
        For each pair $\{x,y\}\subseteq V(H)$, let $S_{x,y}=N_H(x)\cap N_H(y)$. 
        Since $d_H(x),d_H(y)\ge 4\Delta/5$ and $|H|\le \Delta+c$, we have
		\[
		|S_{x,y}|\ge d_H(x)+d_H(y)-|H|\ge \frac{3}{5}\Delta-c.
		\]
		In particular, $|S_{x,y}|\ge 2\Delta/5$ for every pair $\{x,y\}\subseteq V(H)$.
		
		Thus, if $H$ has three disjoint pairs of nonadjacent vertices $(x_1,y_1),\ldots,(x_3,y_3)$, then
		\[
		|S_{x_1,y_1}|+|S_{x_2,y_2}|+|S_{x_3,y_3}|\ge \frac{6}{5}\Delta.
		\]
		Since $|H|\le \Delta+c$, it follows that at least $\frac{1}{10}\Delta-\frac{c}{2}$ vertices of $H$ belong to at least two of the sets $S_{x_i,y_i}$. In particular, there are at least $100$ such vertices.
		Since there are only three possible intersections among the sets $S_{x_i,y_i}$, one of them contains at least two vertices. Let $a$ and $b$ be two vertices in such an intersection. Then $N_H(a)\cap N_H(b)$ contains two disjoint pairs of nonadjacent vertices, say $(x,y)$ and $(v,w)$. The graph $G-V(H)$ is a proper induced subgraph of $G$, and hence it has a $(\Delta-1)$-coloring; fix one and call it $f$. Since every vertex of $H$ has at least $4\Delta/5$ neighbors in $H$, each of $x,y,v,w$ has at most $\Delta/5$ neighbors in $G-V(H)$, and hence at least $4\Delta/5-1$ available colors with respect to $f$. The available color lists of $x$ and $y$ therefore intersect in at least
		\[
		2(4\Delta/5-1)-(\Delta-1)=3\Delta/5-1
		\]
		colors. Choose one such color for both $x$ and $y$. Similarly, the available color lists of $v$ and $w$ intersect in at least $3\Delta/5-1$ colors, so we may choose a common color for $v$ and $w$ different from the color used on $x$ and $y$. Thus we may extend $f$ to a coloring of $G-(V(H)\setminus\{x,y,v,w\})$ in which $x$ and $y$ receive the same color, and $v$ and $w$ receive the same color.
		Let
		\[
		S=\{a,b\}\cup \bigl(N_H(a)\cap N_H(b)\bigr)\setminus\{x,y,v,w\},
		\]
		and let
		\[
		T=\{u\in V(H)\setminus S:\ |N(u)\cap S|\ge 2\}.
		\]
        
		We note that the minimum degree condition ensures that $|S|\ge 2\Delta/5$. If $|S|\le 3\Delta/5-1$, then each vertex of $S$ is adjacent to at least $4\Delta/5-(|S|-1)=4\Delta/5+1-|S|$ vertices in  $H-S$. Hence $e(S,H-S)\ge |S|(4\Delta/5+1-|S|).$ On the other hand, every vertex of $T$ has at most $|S|$ neighbors in $S$, while every vertex of $H-S-T$ has at most one neighbor in $S$. Therefore $e(S,H-S)\le |T||S|+|H-S-T|.$ Since
		\[
		|H-S-T|=|H|-|S|-|T|\le \Delta+c-|S|-|T|,
		\]
		we obtain
		\[
		|S|(4\Delta/5+1-|S|)\le |T||S|+\Delta+c-|S|-|T|.
		\]
		If $|S\cup T|<3\Delta/5$, then $|T|<3\Delta/5-|S|$, and so
		\[
		|S|(4\Delta/5+1-|S|)<|S|(3\Delta/5-|S|)+\Delta+c-|S|.
		\]
		Thus $
		2\Delta/5(\Delta/5+2)<|S|(\Delta/5+2)<\Delta+c,$
		which is impossible when $\Delta\ge 1.0\times 10^6$ and $c\le \Delta/10$. Therefore, we have $|S\cup T|\ge 3\Delta/5$.
Since $\delta(H)\ge 4\Delta/5$ and $|S\cup T|\ge 3\Delta/5$,
		every vertex in $V(H)\setminus (S\cup T)$ has at least $2\Delta/5-c$ neighbors in $S\cup T$. 
		
		Now we color the still uncolored vertices of $H$ in four stages, keeping the already colored vertices $x,y,v,w$ fixed. First, we greedily color the still uncolored vertices of $V(H)\setminus (S\cup T)$. This is possible since each such vertex has at least $2\Delta/5-c$ neighbors in $S\cup T$; even if all four of $x,y,v,w$ lie in $T$, at least $2\Delta/5-c-4\ge2$ of these neighbors are still uncolored. Next, we greedily color the still uncolored vertices of $T$. This is possible since each vertex of $T$ has at least two neighbors in $S$, and all vertices of $S$ are still uncolored at this stage. Then we greedily color the vertices of $S\setminus\{a,b\}$. This is possible because every vertex in $S\setminus\{a,b\}$ is adjacent to both $a$ and $b$, which are still uncolored. Finally, we color $a$ and $b$. Since both $a$ and $b$ are adjacent to $x,y,v,w$, and these four vertices use only two colors, each of $a$ and $b$ has an available color. Thus we obtain a $(\Delta-1)$-coloring of $G$, a contradiction.
		
		So we may assume that $H$ contains no three disjoint pairs of nonadjacent vertices. Let $M$ be a maximum family of pairwise disjoint nonadjacent pairs in $H$. Then $|M|\le 2$. Deleting all vertices contained in pairs of $M$, the remaining graph is a clique; otherwise, there exists another nonadjacent pair that could be added to $M$. Hence $H$ contains a clique of size at least $|V(H)|-4$. In particular, if $C_H$ is a maximum clique of $H$, then $|V(H)\setminus C_H|\le 4$.
		
		Now suppose that $|V(H)\setminus C_H|\ge 2$. Let $u,v\in V(H)\setminus C_H$ be distinct. Since $C_H$ is a maximum clique, there exist $a,b\in C_H$ such that $ua,vb\notin E(H)$. If $a\ne b$, then $\{u,a\}$ and $\{v,b\}$ are two disjoint pairs of nonadjacent vertices. If $a=b$, then $uv\notin E(H)$, for otherwise $(C_H\setminus\{a\})\cup\{u,v\}$ would be a clique larger than $C_H$. Hence $\{a,u,v\}$ is a stable set of size three. Therefore, in either case, there exists a set $X$ of four vertices of $H$ such that $H[X]$ is $2$-colorable and $|X\cap C_H|= 2$.
		
		By the minimality of $G$ and Brooks' theorem, the graph $G-V(H)$ admits a $(\Delta-1)$-coloring. We extend this coloring by first coloring the vertices of $X$ with only two colors. We next justify that at least two vertices of $C_H\setminus X$ are adjacent to every vertex of $X$. The two vertices in $X\cap C_H$ are adjacent to all vertices of $C_H\setminus X$. Each vertex in $X\setminus C_H$ has at most
		\[
		|H|-1-4\Delta/5\le \Delta/5+c-1
		\]
		nonneighbors in $H$, and hence excludes at most this many vertices of $C_H$. Since $|X\setminus C_H|\le2$ and $|C_H|\ge |H|-4\ge4\Delta/5-3$, the number of vertices of $C_H\setminus X$ adjacent to all vertices of $X$ is at least
		\[
		(4\Delta/5-3)-2-2(\Delta/5+c-1)=2\Delta/5-2c-3>2,
		\]
		where we use $c\le \Delta/10$ and $\Delta\ge 5.3\times10^6$. Choose two such vertices, leave them until the end, and greedily color the remaining uncolored vertices of $V(H)\setminus X$. Since $X$ uses only two colors, each of these last two vertices has an available color. Thus we obtain a $(\Delta-1)$-coloring of $G$, a contradiction.
		
		Therefore, $|V(H)\setminus C_H|\le 1$. Hence, $H$ is either a clique, or consists of a clique $C_H$ together with a vertex $v_H$. Furthermore, if $H$ is not a clique, then by Lemma~\ref{p=1} we have $|C_H|\le \Delta-2$. This completes the proof of the lemma.
\end{proof}

Let $\mathcal{C}$ be the set of maximal cliques of $G$ with at least $4\Delta/5+1$ vertices. As a consequence of Lemma~\ref{Hcupv}, we have the following corollary.
\begin{corollary}\label{C1C2}
If two distinct elements $C_1$ and $C_2$ of $\mathcal{C}$ with $|C_1|\le |C_2|$ intersect, then $|C_1\setminus C_2|\le 1$. Moreover, no element of $\mathcal{C}$ intersects two other elements of $\mathcal{C}$.
\end{corollary}
\begin{proof}
Let $C_1,C_2\in\mathcal C$ be distinct and intersecting, and put $H=G[C_1\cup C_2]$. If $z\in C_1\cap C_2$, then $z$ is adjacent to every vertex of $C_1\cup C_2$ except itself; hence $|C_1\cup C_2|\le \Delta+1$. Also every vertex of $H$ has at least $4\Delta/5$ neighbors in its own clique. Since $|H|\le \Delta+1\le \Delta+6$, Lemma~\ref{Hcupv}, applied with $c=6$, implies that $H$ is either a clique or a clique together with one extra vertex. The first case is impossible because $C_1$ and $C_2$ are distinct maximal cliques. Therefore $H=K\cup\{u\}$ for some clique $K$.

The maximal cliques of the near-clique $H=K\cup\{u\}$ are precisely $K$ and $N_K(u)\cup\{u\}$. Since $C_1$ and $C_2$ are maximal cliques of $G$ and are contained in $H$, they are also maximal cliques of $H$; hence they are these two cliques. If $C_1=N_K(u)\cup\{u\}$ and $C_2=K$, then $C_1\setminus C_2=\{u\}$. If $C_1=K$ and $C_2=N_K(u)\cup\{u\}$, then the assumption $|C_1|\le |C_2|$ gives $|K|\le |N_K(u)|+1$, and so $|K\setminus N_K(u)|\le1$; hence again $|C_1\setminus C_2|\le1$.

It remains to rule out triple intersections in the intersection graph of $\mathcal C$. Suppose that some $C\in\mathcal C$ intersects two distinct elements $D,E\in\mathcal C$. For any intersecting pair $P,Q\in\mathcal C$, the degree of a vertex in $P\cap Q$ gives $|P\cup Q|\le\Delta+1$, and hence
\[
|P\cap Q|\ge |P|+|Q|-(\Delta+1)\ge 3\Delta/5+1.
\]
Consequently $(C\cap D)\cap(C\cap E)\ne\emptyset$. Let $z$ be a vertex in this triple intersection and set $H'=G[C\cup D\cup E]$. Then $|H'|\le\Delta+1$ and $\delta(H')\ge4\Delta/5$, so Lemma~\ref{Hcupv} again implies that $H'$ is a clique or a clique together with one extra vertex. The clique case contradicts maximality. In the near-clique case, say $H'=K\cup\{u\}$ with $K$ a clique, the only maximal cliques of $H'$ are $K$ and $N_{H'}(u)\cup\{u\}$. Every member of $\mathcal C$ contained in $H'$ is a maximal clique of $H'$, for otherwise it could be enlarged in $G$. Hence $H'$ contains at most two maximal cliques of $G$ from $\mathcal C$, contradicting the distinctness of $C,D,E$.
\end{proof}

Corollary~\ref{C1C2} allows us to organize the large dense parts of $G$ into essentially disjoint clique-like pieces. Consider the intersection graph on $\mathcal C$, where two members are adjacent when they intersect. By Corollary~\ref{C1C2}, each component has size one or two. For a singleton component $\{C\}$, put $S_i=C$ and $C_i=C$. For a component $\{C,D\}$ of size two, the proof of Corollary~\ref{C1C2} shows that $G[C\cup D]$ is a near-clique $K\cup\{u\}$, where $K$ is one of the two maximal cliques $C,D$, and the other maximal clique is $N_K(u)\cup\{u\}$. We choose $C_i:=K\in\mathcal C$, put $u_i:=u$, and set $S_i=C_i\cup\{u_i\}$. Since the other maximal clique belongs to $\mathcal C$, we have $|N_{C_i}(u_i)|\ge4\Delta/5$, and so $u_i$ has at least $4\Delta/5$ neighbors in $C_i$. These sets are pairwise disjoint. Accordingly, we obtain a partition
$
V(G)=L\cup \bigcup_{i=1}^{\ell}S_i,
$
where
$
L=V(G)\setminus \bigcup_{i=1}^{\ell} S_i,
$
such that each $S_i$ is either
\begin{itemize}
\item a clique $C_i\in\mathcal{C}$, or
\item a set consisting of a clique $C_i\in\mathcal C$ together with a vertex $u_i\in V(G)\setminus C_i$ that has at least \(\frac45\Delta\) neighbors in $C_i$. In the latter case, we call $S_i$ a \emph{near-clique}.
\end{itemize}

We shall need the following results concerning this partition.

\begin{lemma}[Reed \cite{Reed1999}]\label{p-3}
If $v$ is a vertex in some $C_i\in \mathcal{C}$ with $|C_i|=\Delta-p$, then there is at most one neighbor of $v$ outside $C_i$ that has more than $p+3$ neighbors in $C_i$. Furthermore, if $v$ has degree $\Delta-1$, then there is no such neighbor.
\end{lemma}
This partition is the structural framework for the coloring argument in the rest of the paper. The next lemma, due to Reed, explains why this partition is useful. It gives a deterministic criterion under which a partial $(\Delta-1)$-coloring can be extended to the whole graph. Later, in Section~\ref{proper-coloring}, we will use probabilistic tools to construct a partial coloring satisfying exactly these conditions.

The following deterministic extension lemma is the form of Reed's extension step needed here.

\begin{lemma}\label{coloring}
Any partial proper $(\Delta-1)$-coloring of $G$ satisfying the following three conditions can be extended to a $(\Delta-1)$-coloring of $G$.
\begin{enumerate}
\item[$(i)$] for every vertex $v\in L$ there are at least two distinct colors, each appearing on at least two colored vertices in the neighborhood of $v$,
\item[$(ii)$] for each near-clique $S_{i}=C_i\cup\{u_i\}$, there are two uncolored neighbors of $u_{i}$ in $C_{i}$, and
\item[$(iii)$] for every $C_{i}$, there are two uncolored vertices $w_{i}$ and $x_{i}$ of $C_{i}$ whose neighborhoods each contain two distinct colors, each appearing on at least two colored vertices in the current partial coloring.
\end{enumerate}
\end{lemma}
\begin{proof}
We use the following elementary fact. If an uncolored vertex $v$ has two uncolored neighbors, or if the colored part of $N(v)$ contains two distinct repeated colors, then $v$ has an available color from $[\Delta-1]$. Indeed, in the first case at most $d(v)-2\le\Delta-2$ colors appear on the colored neighbors of $v$. In the second case, even if all neighbors of $v$ are colored, the presence of two repeated colors means that at most $d(v)-2\le\Delta-2$ distinct colors appear on $N(v)$.

We now extend the coloring. First consider each near-clique $S_i=C_i\cup\{u_i\}$. If $u_i$ is uncolored, condition $(ii)$ gives two uncolored neighbors of $u_i$ in $C_i$, so the elementary fact allows us to color $u_i$.

Next, for each block $S_i$, reserve the two vertices $w_i,x_i\in C_i$ supplied by condition $(iii)$ and greedily color every other uncolored vertex of $C_i$. Whenever such a vertex is colored, the two reserved vertices are still uncolored and adjacent to it, because $C_i$ is a clique; hence it has an available color. After that, color $w_i$ and $x_i$. The two distinct repeated colors in their neighborhoods, guaranteed by condition $(iii)$, involve vertices that were already colored in the original partial coloring and have not been recolored, so these repeated colors are still present. Thus both reserved vertices have available colors.

Finally, greedily color the uncolored vertices of $L$ in an arbitrary order. For every such vertex $v$, the two distinct repeated colors guaranteed by condition $(i)$ are still present, since the vertices carrying them were colored before the extension procedure began and their colors have not changed. Hence each vertex of $L$ has an available color when it is considered. This yields a full proper $(\Delta-1)$-coloring of $G$.
\end{proof}

For the sake of the probabilistic argument in Section~\ref{proper-coloring}, let $A_v$ be the event that there are fewer than two distinct repeated colors in the neighborhood of $v$. For each near-clique $S_i$, let $E_i$ be the event that condition~$(ii)$ fails on $S_i$. For each $S_i$, let $F_i$ be the event that condition~$(iii)$ fails on $C_i$. We note that if none of the events in the set
\[
\mathcal{E}=\left(\bigcup A_v\right)\cup\left(\bigcup E_i\right)\cup\left(\bigcup F_i\right)
\]
occurs, then the random coloring satisfies conditions~$(i)$, $(ii)$, and~$(iii)$ of Lemma~\ref{coloring}.

\section{\texorpdfstring{A proper coloring of $G$}{A proper coloring of G}}\label{proper-coloring}
In this section, we complete the proof of Theorem~\ref{main} by establishing the probabilistic step outlined earlier. More precisely, we show that with positive probability the random partial coloring avoids all bad events introduced in Section~\ref{decomposition}, and therefore satisfies the hypotheses of Lemma~\ref{coloring}. We begin by recalling the Lov\'asz Local Lemma and McDiarmid's bounded-differences inequality, which will be used in the three probability estimates below.

\begin{theorem}[Lov\'asz Local Lemma {\cite{lll}}]\label{l}
Let $\mathcal{X}$ be a finite set of events such that each $X\in\mathcal{X}$ satisfies:
\begin{itemize}
    \item[$(\romannumeral1)$] $\mathbb{P}(X)\le p$, and
    \item[$(\romannumeral2)$] $X$ is mutually independent of all but at most $d$ other events in $\mathcal{X}$.
\end{itemize}
If $ep(d+1)\le 1$, then with positive probability none of the events in $\mathcal{X}$ occurs.
\end{theorem}

We now estimate the probabilities of the three types of bad events corresponding to conditions $(i)$-$(iii)$ of Lemma~\ref{coloring}. We treat these events separately.

We expose a uniformly random coloring
\[
  \varphi:V(G)\to [q],\qquad q:=\Delta-1,
\]
and then uncolor every vertex that has the same color as one of its neighbors. This is the conflict-deletion step.

\subsection{Preliminary probabilistic tools}

\begin{lemma}\label{lem:supplied-triples}
For any $C_i$, if $|C_i|=\Delta-p$, then we can find at least $(\Delta-2)/11$ pairwise disjoint triples, each consisting of a vertex $v\in C_i$ and two neighbors of $v$ outside $C_i$, both of which have at most $p+3$ neighbors in $C_i$.
\end{lemma}

\begin{proof}
Take a maximal family of pairwise disjoint triples of the desired form, and let its size be $k$. Let $T$ be the set of vertices of $C_i$ contained in these triples, and let $S$ be the set of vertices outside $C_i$ contained in these triples. Then $|T|=k$ and $|S|=2k$.
\begin{claim}
Every vertex of $C_i\setminus T$ has at least $p-1$ neighbors in $S$.
\end{claim}
Let $v\in C_i\setminus T$. Since $|C_i|=\Delta-p$ and $\delta(G)\ge \Delta-1$, the vertex $v$ has at least $p$ neighbors outside $C_i$. We claim that at least $p$ of these outside neighbors are good, where a good neighbor means one with at most $p+3$ neighbors in $C_i$. If $d(v)=\Delta-1$, then Lemma~\ref{p-3} says that none of the outside neighbors is bad, and the claim follows. If $d(v)=\Delta$, then $v$ has at least $p+1$ outside neighbors, while Lemma~\ref{p-3} says that at most one of them is bad; again at least $p$ outside neighbors are good. By the maximality of our family of triples, at most one good neighbor of $v$ lies outside $S$, since otherwise we could choose two good neighbors of $v$ outside $S$ and add a new triple through $v$, contrary to maximality. Hence at least $p-1$ good neighbors of $v$ lie in $S$. That is, $v$ has at least $p-1$ neighbors in $S$.  This implies that the number of edges between $S$ and $C_i\setminus T$ is at least
\begin{align*}
    (p-1)(|C_i|-k)=(p-1)(\Delta-p-k).
\end{align*}
On the other hand, each vertex of $S$ has at most $p+3$ neighbors in $C_i$, and hence at most $p+3$ neighbors in $C_i\setminus T$. Since $|S|=2k$, there are at most $2k(p+3)$ edges between $S$ and $C_i\setminus T$. Therefore $(p-1)(\Delta-p-k)\le 2k(p+3)$, and so
		\begin{align*}
		    k\ge \frac{(p-1)(\Delta-p)}{3p+5}.
		\end{align*}
If $p\ge 2$, then this implies $k\ge (\Delta-2)/11$, as required. Indeed, after clearing denominators the desired inequality is equivalent to
\[
8(p-2)\Delta-11p^2+17p+10\ge0,
\]
which is equality for $p=2$ and is positive for $p\ge3$ in the range $p\le\Delta/5-1$.

 If $p=1$, then by Lemma~\ref{p=1}, at most four vertices of $C_i$ have degree $\Delta-1$. Hence at least $\Delta-5$ vertices of $C_i$ have degree $\Delta$. For each such vertex $v$, since $|C_i|=\Delta-1$, the vertex $v$ has exactly two neighbors outside $C_i$, and both of them are good by Lemma~\ref{p=1}. Thus these vertices of $C_i$ are all candidates for the desired triples.

Now choose pairwise disjoint triples greedily. Each chosen triple $\{v,x,y\}$ blocks at most seven candidate vertices: namely, $v$ itself, together with at most three further candidates adjacent to $x$ and at most three further candidates adjacent to $y$, since each of $x$ and $y$ has at most four neighbors in $C_i$. Therefore, if the maximal family has size $k$, then $7k\ge \Delta-5$. Hence $k\ge (\Delta-5)/7\ge(\Delta-2)/11$ for $\Delta\ge 11$.
This completes the proof of Lemma~\ref{lem:supplied-triples}.
\end{proof}

\begin{lemma}\label{lem:many-nonedges}
For every vertex $v\in L$, the neighborhood $N(v)$ contains at least $
   \frac{\Delta^2}{50}-\Delta$
unordered pairs of nonadjacent vertices.
\end{lemma}

\begin{proof}
Suppose $v\in L$. We iteratively choose vertices
\[
  v_0,v_1,\ldots,v_{M-1}\in N(v),\qquad M:=\left\lfloor\frac{\Delta}{5}\right\rfloor,
\]
such that, for each $i$, the vertex $v_i$ has fewer than $4\Delta/5$ neighbors in
\[
  R_i:=N(v)\setminus\{v_0,\ldots,v_{i-1}\}.
\]
Indeed, if no such vertex existed and $|R_i|\ge4\Delta/5+1$, then $G[R_i]$ would satisfy the hypothesis of Lemma~\ref{Hcupv}; hence $G[R_i]$ would be a clique or a near-clique. If it were a clique, then $R_i\cup\{v\}$ would contain a clique of size at least $4\Delta/5+2$ containing $v$. If it were a near-clique, its clique part would have size at least $|R_i|-1\ge4\Delta/5$, and adjoining $v$ to that clique would give a clique of size at least $4\Delta/5+1$ containing $v$. In either case $v$ lies in a member of $\mathcal C$, and therefore in some block $S_j$, contrary to $v\in L$. If $|R_i|\le4\Delta/5$, then every vertex of $R_i$ has fewer than $4\Delta/5$ neighbors in $R_i$, so the choice is immediate. The construction is possible for all $i<M$ because $d(v)\ge\Delta-1$ by Observation~\ref{obs:critical}.

For $0\le i<M$, the set $R_i$ has size at least $\Delta-1-i$. Since $v_i$ has fewer than $4\Delta/5$ neighbors in $R_i$, it has more than
\[
  |R_i|-1-\frac{4\Delta}{5}\ge \frac{\Delta}{5}-i-2
\]
non-neighbors in $R_i$. Counting only the nonedges first exposed at step $i$, we obtain
\[
\begin{aligned}
  |
\{\{x,y\}\subseteq N(v):xy\notin E(G)\}|
  &\ge \sum_{i=0}^{M-1}\left(\frac{\Delta}{5}-i-2\right)  \\
  &= M\left(\frac{\Delta}{5}-2\right)-\frac{M(M-1)}2.
\end{aligned}
\]
The function $x\mapsto x(\Delta/5-2)-x(x-1)/2$ is decreasing on the interval $[\Delta/5-1,\Delta/5]$ when $\Delta\ge10$. Since $\Delta/5-1\le M\le\Delta/5$, the last expression is at least
\[
  \frac{\Delta}{5}\left(\frac{\Delta}{5}-2\right)-\frac{(\Delta/5)(\Delta/5-1)}2
  =\frac{\Delta^2}{50}-\frac{3\Delta}{10}
  \ge \frac{\Delta^2}{50}-\Delta.
\]
This proves the lemma.
\end{proof}

\begin{lemma}[McDiarmid's inequality \cite{McDiarmid1989}]\label{lem:mcdiarmid}
Let $X=f(Z_1,\ldots,Z_N)$, where $Z_1,\ldots,Z_N$ are independent random variables. Suppose changing the $j$th coordinate can change $X$ by at most $c_j$. Then, for every $s>0$,
\[
  \mathbb P(X\le \mathbb{E} X-s)\le \exp\left(-\frac{2s^2}{\sum_j c_j^2}\right),
\]
and similarly
\[
  \mathbb P(X\ge \mathbb{E} X+s)\le \exp\left(-\frac{2s^2}{\sum_j c_j^2}\right).
\]
\end{lemma}

\begin{lemma}[Chernoff bound for negatively associated indicators \cite{DubhashiRanjan1998}]\label{lem:chernoff-na}
Let $X=\sum_j X_j$ be a sum of Bernoulli random variables with mean $\mu$. If $(X_j)$ are independent, or more generally negatively associated, then for $0<a<\mu$,
\[
  \mathbb P(X<a)\le \exp\left(-\frac{(\mu-a)^2}{2\mu}\right).
\]
\end{lemma}

\begin{lemma}[Closure properties of negative association \cite{DubhashiRanjan1998,JoagDevProschan1983}]\label{lem:NA-closure}
For a fixed vertex $z$, the one-hot color indicators
\[
  X_{z,\gamma}:=\bbone_{\{\varphi(z)=\gamma\}},\qquad \gamma\in[q],
\]
are negatively associated. Independent unions of negatively associated families are negatively associated. Finally, applying coordinatewise monotone functions, all in the same direction, to pairwise disjoint subfamilies preserves negative association; in particular this applies to all nondecreasing functions and also to all nonincreasing functions.
\end{lemma}

We shall repeatedly use the elementary inequalities
\begin{align}
  \log(1-x)&\ge -\frac{x}{1-x}, &&0<x<1, \tag{1}\label{eq:log1}\\
  \log(1-x)&\ge -x-x^2, &&0\le x\le \frac12. \tag{2}\label{eq:log2}
\end{align}

\subsection{The leftover-vertex events \texorpdfstring{$A_v$}{Av}}\label{sec:leftover}

Recall that $V(G)=S_1\cup\cdots\cup S_\ell\cup L$,
where $L$ is the set of vertices outside all cliques and near-cliques in the decomposition from Section~\ref{decomposition}. Fix a vertex $v\in L$. We expose a uniformly random $(\Delta-1)$-coloring of $V(G)$, and then uncolor every vertex that has the same color as one of its neighbors.
Now we consider the events $A_v$. Set $\Delta_0=5.3 \times10^6.$

\begin{lemma}\label{lem:Av}
For every $v\in L$ and every $\Delta\ge \Delta_0$,
$
    \mathbb{P}(A_v)
    \le
    \exp(-2.43\cdot 10^{-5}\Delta)
    +
    \exp(-2.2\cdot 10^{-5}\Delta).
$
\end{lemma}

\begin{proof}
Fix $v\in L$ and let $H_v$ be the number of colors $\gamma\in[q]$ such that
$\gamma$ appears exactly twice on $N(v)$ and the two vertices of $N(v)$
receiving color $\gamma$ are nonadjacent.

\begin{claim}\label{claim:Hv-exp}
For every $v\in L$ and every $\Delta\ge \Delta_0$,
$
    \EE H_v\ge 0.00735\Delta.
$
\end{claim}

\begin{proof}[Proof of Claim~\ref{claim:Hv-exp}]
Let $d=|N(v)|\le \Delta$, and let $P_v$ be the set of unordered nonadjacent
pairs in $N(v)$.  By Lemma~\ref{lem:many-nonedges},
\[
    |P_v|\ge  \frac{\Delta^2}{50}-\Delta.
\]
For a fixed pair $(x,y)\in P_v$, the probability that $x$ and $y$ receive the
same color and that no other vertex of $N(v)$ receives this color is
\[
    \frac{1}{q}\left(1-\frac{1}{q}\right)^{d-2}.
\]
Since $d\le \Delta$ and $q=\Delta-1$,
\[
    \left(1-\frac{1}{q}\right)^{d-2}
    \ge
    \left(1-\frac{1}{q}\right)^{\Delta-2}
    =
    \left(1-\frac{1}{q}\right)^{q-1}
    \ge e^{-1}.
\]
Each color counted by $H_v$ corresponds to a unique pair in $P_v$.  Hence, by
linearity of expectation,
\[
    \EE H_v
    \ge
    \left( \frac{\Delta^2}{50}-\Delta\right)
    \frac{e^{-1}}{\Delta-1}
    =
    \frac{\Delta(\Delta-50)}{50(\Delta-1)}e^{-1}.
\]
For $\Delta\ge \Delta_0$, the right-hand side is at least $0.00735\Delta$.
\end{proof}

\begin{claim}\label{claim:Hv-conc}
For every $v\in L$ and every $\Delta\ge \Delta_0$,
$
    \mathbb{P}(H_v<3.7\cdot 10^{-4}\Delta)
    \le
    \exp(-2.43\cdot 10^{-5}\Delta).
$
\end{claim}

\begin{proof}[Proof of Claim~\ref{claim:Hv-conc}]
The random variable $H_v$ depends only on the colors of the vertices in
$N(v)$.  Changing the color of one vertex can affect only the status of the
old color and the new color, and each fixed color can enter or leave the count
by at most one.  Thus $H_v$ is $2$-Lipschitz in each of at most $\Delta$
coordinates.  By Claim~\ref{claim:Hv-exp},
\[
    \EE H_v-3.7\cdot 10^{-4}\Delta
    \ge
    0.00698\Delta.
\]
McDiarmid's inequality gives
\[
    \mathbb{P}(H_v<3.7\cdot 10^{-4}\Delta)
    \le
    \exp\!\left(
        -\frac{2(0.00698\Delta)^2}{4\Delta}
    \right)
    \le
    \exp(-2.43\cdot 10^{-5}\Delta).
\]\end{proof}

We next condition on a coloring of $N(v)$ for which
$
    H_v\ge 3.7\cdot 10^{-4}\Delta.
$
Choose, in a deterministic way, a family
\[
    Q_v=\{(x_j,y_j,\gamma_j):1\le j\le L_A\},
    \qquad
    L_A:=\lfloor 3.6\cdot 10^{-4}\Delta\rfloor,
\]
where $x_j,y_j\in N(v)$ are nonadjacent, both have color $\gamma_j$, and no
other vertex of $N(v)$ has color $\gamma_j$.  The colors $\gamma_j$ are
pairwise distinct.

\begin{claim}\label{claim:Av-survival}
For every fixed coloring of $N(v)$ with $H_v\ge 3.7\cdot 10^{-4}\Delta$, after $Q_v$ is chosen deterministically as above, the conditional probability over the colors outside $N(v)$ that $A_v$ occurs is at most
$
    \exp(-2.2\cdot 10^{-5}\Delta).
$
\end{claim}

\begin{proof}[Proof of Claim~\ref{claim:Av-survival}]
Expose all colors outside $N(v)$.  For $1\le j\le L_A$, let $I_j$ be the
indicator of the event that the pair $x_j,y_j$ survives conflict deletion
with color $\gamma_j$.  Since $x_jy_j\notin E(G)$ and no other vertex of
$N(v)$ has color $\gamma_j$, this pair can be killed only by a vertex outside
$N(v)$ adjacent to $x_j$ or $y_j$ that receives color $\gamma_j$.  There are at
most $2\Delta$ such possible blockers.  Hence
\[
    \mathbb{P}(I_j=1)
    \ge
    \left(1-\frac{1}{q}\right)^{2\Delta}.
\]
By \eqref{eq:log1},
\[
    \left(1-\frac{1}{q}\right)^{2\Delta}
    \ge
    \exp\!\left(-\frac{2\Delta}{q-1}\right)
    =
    \exp\!\left(-\frac{2\Delta}{\Delta-2}\right)
    \ge 0.135
\]
for $\Delta\ge \Delta_0$.

The indicators $I_j$ are negatively associated.  Indeed, for each unexposed
vertex $z$, the one-hot indicators $X_{z,\gamma_j}$ are negatively
associated, and independent unions over $z$ remain negatively associated.
Moreover, each $I_j$ is a coordinatewise decreasing function of the variables
with target color $\gamma_j$.  Since the colors $\gamma_j$ are pairwise
distinct, the variable subfamilies belonging to different $j$ are pairwise
disjoint.  Lemma~\ref{lem:NA-closure} applies.

Let
\[
    S_A:=\sum_{j=1}^{L_A} I_j.
\]
Then
\[
    \mu_A:=\EE S_A
    \ge
    0.135\lfloor 3.6\cdot 10^{-4}\Delta\rfloor
    \ge
    4.86\cdot 10^{-5}\Delta-0.135.
\]
For $\Delta\ge \Delta_0$, this is larger than $2$.  By
Lemma~\ref{lem:chernoff-na},
\[
    \mathbb{P}(S_A<2)
    \le
    \exp\!\left(-\frac{(\mu_A-2)^2}{2\mu_A}\right)
    \le
    \exp(-2.2\cdot 10^{-5}\Delta),
\]
where the last inequality is checked at $\Delta_0$ and then follows by
monotonicity.  If $S_A\ge 2$, then $N(v)$ contains at least two repeated
colors after conflict deletion.  Hence $A_v\subseteq\{S_A<2\}$ under the
conditioning.
\end{proof}

Combining Claims~\ref{claim:Hv-conc} and~\ref{claim:Av-survival}, we obtain
\[
\begin{aligned}
    \mathbb{P}(A_v)
    &\le
    \mathbb{P}(H_v<3.7\cdot 10^{-4}\Delta)
    +
    \mathbb{P}(A_v\mid H_v\ge 3.7\cdot 10^{-4}\Delta)  \\
    &\le
    \exp(-2.43\cdot 10^{-5}\Delta)
    +
    \exp(-2.2\cdot 10^{-5}\Delta).
\end{aligned}
\]
This completes the proof.
\end{proof}

\subsection{The near-clique events \texorpdfstring{$E_i$}{Ei}}\label{sec:nearclique}
We next turn to the bad events associated with near-cliques.
\begin{lemma}\label{lem:Ei}
For every near-clique index $i$ and every $\Delta\ge \Delta_0$, $    \mathbb{P}(E_i)\le \exp(-0.15\Delta).$
\end{lemma}

\begin{proof}
Fix a set $R_i\subseteq N(u_i)\cap C_i$ with $  r:=|R_i|=\left\lfloor\frac{4\Delta}{5}\right\rfloor.$
Let $X_i$ be the number of distinct colors appearing on $R_i$ before conflict deletion. If two vertices of $R_i$ receive the same color, then they are adjacent in the clique $C_i$, and both are uncolored by conflict deletion. Thus $u_i$ has at least two uncolored neighbors in $C_i$. Consequently, $
  E_i\subseteq \{X_i=r\}. $

For each color $\gamma\in[q]$, the probability that $\gamma$ appears on $R_i$ is $1-(1-1/q)^r$. Hence
\[
  \mathbb{E}X_i=q\left(1-\left(1-\frac1q\right)^r\right),
\]
and
\[
  r-\mathbb{E}X_i=r-q+q\left(1-\frac1q\right)^r.
\]
Using $r\ge 4\Delta/5-1$, $r\le 4\Delta/5$, and \eqref{eq:log1},
\[
  r-\mathbb{E}X_i
  \ge -\frac{\Delta}{5}+(\Delta-1)\exp\left(-\frac{4\Delta}{5(\Delta-2)}\right).
\]
The ratio
\[
  \left(1-\frac1\Delta\right)\exp\left(-\frac{4\Delta}{5(\Delta-2)}\right)
\]
is increasing for $\Delta\ge \Delta_0$, because the derivative of its logarithm is
\[
  \frac{1}{\Delta(\Delta-1)}+\frac{8}{5(\Delta-2)^2}>0.
\]
At $\Delta=\Delta_0$ it is greater than $0.449$, and hence
\[
  r-\mathbb{E}X_i\ge 0.249\Delta>0.245\Delta. \tag{6}
\]
Changing one color on $R_i$ can change $X_i$ by at most one. By Lemma~\ref{lem:mcdiarmid} and (6),
\[
  \mathbb P(E_i)\le \mathbb P(X_i=r)
  \le \mathbb P(X_i\ge \mathbb{E}X_i+0.245\Delta)
  \le \exp\left(-\frac{2(0.245\Delta)^2}{r}\right)
  \le \exp(-0.15\Delta).
\]
\end{proof}

\subsection{The large-clique events \texorpdfstring{$F_i$}{Fi}}\label{subsec:large-clique}

Fix a large clique $C=C_i$ and write
\[
  |C|=n=\Delta-p,
  \qquad 1\le p\le \frac{\Delta}{5}-1.
\]
We prove the large-clique estimate
\[
  \mathbb{P}(F_i)\le e^{-3.5\cdot10^{-5}\Delta}+e^{-1.7\cdot10^{-5}\Delta}+e^{-2.1\cdot10^{-5}\Delta}+e^{-1.5\cdot10^{-5}\Delta}.
\]

Expose only the colors of vertices of $C$.
A color is called a
\emph{singleton color on $C$} if it appears exactly once on $C$.
By Lemma~\ref{lem:supplied-triples}, fix a family
$
  \mathcal{T}=\{T_t=(a_t,b_t,c_t):t\in[m]\}
$
of pairwise disjoint supplied triples with $m\ge(\Delta-2)/11$. Thus $c_t\in C$, $a_t,b_t\notin C$, $a_t,b_t\in N(c_t)$, and
$
  |N_C(a_t)|, |N_C(b_t)|\le p+3.
$
A supplied
triple $(a_t,b_t,c_t)$ is called \emph{internally duplicate} if the color of
$c_t$ appears on at least one other vertex of $C$.
\begin{lemma}\label{lem:endpoint-singletons}
Expose the colors of vertices of $C$ first. For each supplied endpoint
$
  z\in \{a_t,b_t:t\in[m]\},
$
let
$
  S_z=\{\gamma\in[q]:\gamma\text{ appears exactly once on }C,\text{ and its unique vertex in }C\text{ is not adjacent to }z\}.
$
Then, for every $\Delta\ge\Delta_0$,
\[
  \mathbb{P}\bigl(\exists z: |S_z|<0.26\Delta\bigr)
  \le \exp(-3.5\cdot10^{-5}\Delta).
\]
\end{lemma}

\begin{proof}
Fix a supplied endpoint $z$. Since $|N_C(z)|\le p+3$, the vertex $z$ has at least
$
  |C|-(p+3)=\Delta-2p-3
$
non-neighbors in $C$. Each such non-neighbor contributes to $S_z$ if its color appears nowhere else on $C$. Therefore
\[
  \mathbb{E} |S_z|\ge (\Delta-2p-3)\left(1-\frac1q\right)^{\Delta-p-1}. \tag{7}
\]
This is a lower bound, not an equality, because $z$ may have more than $\Delta-2p-3$ non-neighbors in $C$.

Let
\[
  g(p):=(\Delta-2p-3)\left(1-\frac1{\Delta-1}\right)^{\Delta-p-1}.
\]
For $1\le p\le \Delta/5-1$,
\[
  \frac{d}{dp}\log g(p)
  =-\frac{2}{\Delta-2p-3}-\log\left(1-\frac1{\Delta-1}\right).
\]
By \eqref{eq:log1},
\[
  -\log\left(1-\frac1{\Delta-1}\right)\le \frac1{\Delta-2},
\]
and since $p\ge1$,
\[
  -\frac{2}{\Delta-2p-3}\le -\frac2{\Delta-5}.
\]
Thus
\[
  \frac{d}{dp}\log g(p)
  \le -\frac2{\Delta-5}+\frac1{\Delta-2}<0.
\]
Hence $g$ is decreasing on the permitted interval, and the worst case is $p=\Delta/5-1$. By (7),
\[
  \mathbb{E} |S_z|
  \ge \left(\frac{3\Delta}{5}-1\right)
       \left(1-\frac1{\Delta-1}\right)^{4\Delta/5}.
\]
Again using \eqref{eq:log1},
\[
  \left(1-\frac1{\Delta-1}\right)^{4\Delta/5}
  \ge \exp\left(-\frac{4\Delta}{5(\Delta-2)}\right)
  \ge e^{-0.800001}
\]
for $\Delta\ge\Delta_0$. Consequently
\[
  \mathbb{E} |S_z|>0.269\Delta. \tag{8}
\]

Changing one color assigned to a vertex of $C$ can affect only the old color and the new color, so $|S_z|$ is $2$-Lipschitz in each of the $n\le\Delta$ coordinates. By Lemma~\ref{lem:mcdiarmid} and (8),
\[
  \mathbb{P}(|S_z|<0.26\Delta)
  \le \exp\left(-\frac{2(0.009\Delta)^2}{4\Delta}\right)
  \le \exp(-4.0\cdot10^{-5}\Delta).
\]
There are at most $2m\le 2\Delta$ supplied endpoints. Since
\[
  2\Delta\exp(-4.0\cdot10^{-5}\Delta)
  \le \exp(-3.5\cdot10^{-5}\Delta)
\]
for $\Delta\ge\Delta_0$, the union bound proves the lemma.
\end{proof}

\begin{lemma}\label{lem:many-duplicate-centers}
 Let $D_C$ be the number of internally duplicate triples in $\mathcal{T}$, and put
$
  L_F:=\lfloor0.038\Delta\rfloor.
$
Then, for every $\Delta\ge\Delta_0$,
\[
  \mathbb{P}(D_C<L_F)\le \exp(-1.7\cdot10^{-5}\Delta).
\]
\end{lemma}

\begin{proof}
For a fixed center $c_t$,
\[
  \mathbb{P}(t\text{ is internally duplicate})
  =1-\left(1-\frac1q\right)^{|C|-1}.
\]
Since $|C|-1=\Delta-p-1\ge 4\Delta/5$, we have
\[
  \mathbb{P}(t\text{ is internally duplicate})
  \ge 1-\exp\left(-\frac{4\Delta}{5(\Delta-1)}\right)>0.55.
\]
Therefore
\[
  \mathbb{E} D_C\ge 0.55m\ge0.55\cdot\frac{\Delta-2}{11}=0.05\Delta-0.1. \tag{9}
\]
In particular, since $L_F\le0.038\Delta$ and $\Delta\ge\Delta_0$,
\[
  \mathbb E D_C-L_F\ge 0.012\Delta-0.1\ge0.0119\Delta. \tag{10}
\]

Changing one color inside $C$ from $\alpha$ to $\beta$ can affect duplicate status only for centers whose colors are $\alpha$ or $\beta$. For one fixed color, the duplicate status can change only when the corresponding color class crosses between size $1$ and size $2$; at that moment the color class contains at most two vertices, and because the supplied triples are pairwise disjoint, at most two selected centers. Thus $D_C$ is $4$-Lipschitz in each color on $C$. By Lemma~\ref{lem:mcdiarmid} and (10),
\[
  \mathbb{P}(D_C<L_F)
  \le \exp\left(-\frac{2(0.0119\Delta)^2}{16\Delta}\right)
  \le \exp(-1.7\cdot10^{-5}\Delta).
\]
\end{proof}

Let
\[
  G_1:=\{\forall z,\ |S_z|\ge0.26\Delta\},\qquad
  G_2:=\{D_C\ge L_F\}.
\]
By Lemmas~\ref{lem:endpoint-singletons} and~\ref{lem:many-duplicate-centers},
\[
  \mathbb{P}(G_1^c)\le e^{-3.5\cdot10^{-5}\Delta},\qquad
  \mathbb{P}(G_2^c)\le e^{-1.7\cdot10^{-5}\Delta}. \tag{11}
\]

\begin{lemma}\label{lem:clean-active}
Condition on $G_1\cap G_2$. From the internally duplicate triples, choose deterministically a subfamily
$
  B\subseteq [m],\ |B|=L_F=\lfloor0.038\Delta\rfloor.
$
For $t\in B$, let
$
  \alpha_t:=\varphi(a_t),\beta_t:=\varphi(b_t).
$
Call $t$ clean active if
$
  \alpha_t\in S_{a_t}, \beta_t\in S_{b_t}, \alpha_t\ne \beta_t,
$
and neither $\alpha_t$ nor $\beta_t$ appears among the other $2L_F-2$ endpoint colors $\{\alpha_s,\beta_s:s\in B, s\ne t\}$. Let $Y$ be the number of clean active triples in $B$, and set $
  r:=\lfloor3.5\cdot10^{-4}\Delta\rfloor.
$
Then
\[
  \mathbb{P}(Y<r\mid G_1\cap G_2)
  \le \exp(-2.1\cdot10^{-5}\Delta).
\]
Consequently, with complementary probability there exists a clean active subfamily $R\subseteq B$ with $|R|=r$.
\end{lemma}

\begin{proof}
The events $G_1$ and $G_2$, and the deterministic choice of $B=B(\varphi|_C)$, depend only on the colors of vertices of $C$. The supplied triples are pairwise vertex-disjoint, so the endpoint variables indexed by $B$ are distinct variables outside $C$. Therefore these endpoint colors remain mutually independent and uniform on $[q]$ after conditioning on $G_1\cap G_2$. Fix $t\in B$. By $G_1$,
$
  |S_{a_t}|, |S_{b_t}|\ge0.26\Delta.
$
Hence
\begin{equation}\tag{12}
\begin{aligned}
  \mathbb{P}(t\text{ is clean active})
  &\ge \left[\left(\frac{0.26\Delta}{q}\right)^2-\frac1q\right]
       \left(1-\frac2q\right)^{2L_F-2}.
\end{aligned}
\end{equation}
The term $-1/q$ accounts for the possible event $\alpha_t=\beta_t$. For $\Delta\ge\Delta_0$,
\[
  \left(\frac{0.26\Delta}{q}\right)^2-\frac1q>0.0675. \tag{13}
\]
Moreover $2L_F-2\le0.076\Delta$, and by \eqref{eq:log1},
\[
  \left(1-\frac2q\right)^{2L_F-2}
  \ge \exp\left(-\frac{0.152\Delta}{\Delta-3}\right)>0.858. \tag{14}
\]
Combining (12)--(14),
\[
  \mathbb{P}(t\text{ is clean active})>0.057,
\]
and therefore
\[
  \mathbb{E} Y\ge0.057L_F. \tag{15}
\]
Since $L_F\ge0.038\Delta-1$ and $r\le3.5\cdot10^{-4}\Delta$,
\[
  \mathbb{E} Y-r
  \ge0.057(0.038\Delta-1)-3.5\cdot10^{-4}\Delta
  \ge0.0018\Delta \tag{16}
\]
for $\Delta\ge\Delta_0$.

We next justify the Lipschitz bound used for $Y$.
\begin{claim}\label{claim:Y-lipschitz}
Changing one endpoint color changes $Y$ by at most $2$.
\end{claim}
\begin{proof}
Suppose that one endpoint color is changed from $\alpha$ to $\beta$. Apart from the triple containing the changed endpoint, only triples containing another endpoint of color $\alpha$ or $\beta$ can change their clean-active status. Since a clean-active triple requires each of its two endpoint colors to be unique among all endpoint colors in $B$, at most one other triple can become clean active because the old color $\alpha$ becomes unique, and at most one other clean-active triple can be destroyed because the new color $\beta$ ceases to be unique. If the changed triple becomes clean active after the change, then $\beta$ is unique after the change, so no other clean-active triple can be destroyed by the color $\beta$. If the changed triple ceases to be clean active, then $\alpha$ was unique before the change, so no other triple can be created by the color $\alpha$. Finally, if the changed triple does not change status, the possible outside creation and outside destruction have opposite signs, and hence the net change has absolute value at most $1$. In all cases, the absolute change in $Y$ is at most $2$.
\end{proof}
By Claim~\ref{claim:Y-lipschitz}, $Y$ is $2$-Lipschitz in each of the $2L_F$ endpoint colors. Hence
\[
  \sum_j c_j^2\le 2L_F\cdot2^2=8L_F\le0.304\Delta.
\]
By Lemma~\ref{lem:mcdiarmid} and (16),
\[
  \mathbb{P}(Y<r\mid G_1\cap G_2)
  \le \exp\left(-\frac{2(0.0018\Delta)^2}{0.304\Delta}\right)
  \le \exp(-2.1\cdot10^{-5}\Delta).
\]
If $Y\ge r$, choose any $r$ clean active triples; by definition, all $2r$ target colors $\{\alpha_t,\beta_t:t\in R\}$ are pairwise distinct.
\end{proof}

\begin{lemma}\label{lem:survival}
Condition on $G_1\cap G_2$ and on a clean active family $R$ with $|R|=r$. Let $S_F$ be the number of triples in $R$ which survive in the sense that both associated repeated-color nonedges survive conflict deletion. Then
\[
  \mathbb{P}(S_F<2\mid G_1\cap G_2, R)
  \le \exp(-1.5\cdot10^{-5}\Delta).
\]
\end{lemma}

\begin{proof}
Fix $t\in R$. Let $x_t\in C$ be the unique vertex of color $\alpha_t$ in $C$, and let $y_t\in C$ be the unique vertex of color $\beta_t$ in $C$. Since $\alpha_t\in S_{a_t}$ and $\beta_t\in S_{b_t}$,
\[
  a_tx_t\notin E(G),\qquad b_ty_t\notin E(G). \tag{17}
\]
Because $t$ is internally duplicate, the color of $c_t$ is not a singleton color on $C$. Hence $x_t,y_t\ne c_t$. Since $C$ is a clique,
$
  x_t,y_t\in N(c_t).
$
Also $a_t,b_t\in N(c_t)$ by the definition of supplied triples. Thus $(a_t,x_t)$ and $(b_t,y_t)$ are two same-color nonedges inside $N(c_t)$.

The pair $(a_t,x_t)$ can be killed only if some vertex adjacent to $a_t$ or $x_t$ receives color $\alpha_t$. Inside $C$ there is no such blocker: $\alpha_t$ is a singleton color on $C$, and the unique vertex $x_t$ is nonadjacent to $a_t$. Outside $C$, the vertex $a_t$ has at most $\Delta$ neighbors, and $x_t\in C$ has at most $p+1\le \Delta/5$ outside neighbors. Hence this pair has at most $1.2\Delta$ blockers. The same bound holds for $(b_t,y_t)$, so the total blocker-list multiplicity for both pairs is at most $2.4\Delta$.

All endpoints in $B\setminus\{t\}$ avoid the target colors $\alpha_t,\beta_t$ by clean activity. Thus any blocker already exposed among these endpoints does not kill either pair, and all remaining possible blockers still have independent uniform colors. Let $M$ be the number of remaining blockers relevant to exactly one of the two target colors $\alpha_t,\beta_t$, and let $N$ be the number of remaining blockers relevant to both target colors. The blocker-list multiplicity bound gives
$
  M+2N\le 2.4\Delta.
$
A one-color blocker avoids its target color with probability $1-1/q$, while a two-color blocker avoids both target colors with probability $1-2/q$. Since
\[
  1-\frac1q\ge \left(1-\frac2q\right)^{1/2},
\]
we have
\[
\begin{aligned}
  \mathbb{P}(t\text{ survives})
  &\ge (1-1/q)^M(1-2/q)^N  \\
  &\ge (1-2/q)^{M/2+N}
   \ge \left(1-\frac2q\right)^{1.2\Delta}.
\end{aligned}
\]
We shall use the slightly weaker bound
\[
  \mathbb{P}(t\text{ survives})
  \ge \left(1-\frac2q\right)^{1.2\Delta+1}. \tag{18}
\]
By \eqref{eq:log1}, for $\Delta\ge\Delta_0$,
\[
  \left(1-\frac2q\right)^{1.2\Delta+1}
  \ge \exp\left(-\frac{2.4\Delta+2}{\Delta-3}\right)>0.0906. \tag{19}
\]
Let $I_t$ be the indicator that $t$ survives. Since the target colors belonging to different triples of $R$ are pairwise distinct, the variable subfamilies used to determine different $I_t$ are pairwise disjoint. For each unexposed vertex $z$, the one-hot indicators $\bbone_{\{\varphi(z)=\gamma\}}$ over the target colors are negatively associated, and the survival indicators are coordinatewise decreasing functions of disjoint target-color subfamilies. Lemma~\ref{lem:NA-closure} therefore implies that $(I_t)_{t\in R}$ is negatively associated.

Let
\[
  S_F:=\sum_{t\in R} I_t,
\]
and let $\mu_F:=\mathbb{E} S_F$. By (19),
\[
  \mu_F\ge0.0906r\ge0.0906(3.5\cdot10^{-4}\Delta-1)=: \mu_0(\Delta).
\]
At $\Delta=\Delta_0$, $\mu_0(\Delta)>142.6$. Since the function
\[
  \mu\mapsto \frac{(\mu-2)^2}{2\mu}=\frac\mu2-2+\frac2\mu
\]
is increasing for $\mu>2$, Lemma~\ref{lem:chernoff-na} gives
\[
  \mathbb{P}(S_F<2)
  \le \exp\left(-\frac{(\mu_F-2)^2}{2\mu_F}\right)
  \le \exp\left(-\frac{(\mu_0-2)^2}{2\mu_0}\right).
\]
Finally, the function
\[
  \Delta\mapsto \frac{(\mu_0(\Delta)-2)^2}{2\mu_0(\Delta)}-1.5\cdot10^{-5}\Delta
\]
is increasing for $\Delta\ge\Delta_0$. Indeed, with $a:=0.0906\cdot3.5\cdot10^{-4}=3.171\cdot10^{-5}$, we have $\mu_0(\Delta)=a\Delta-0.0906$ and
\[
  \frac{d}{d\Delta}\left(\frac{(\mu_0-2)^2}{2\mu_0}-1.5\cdot10^{-5}\Delta\right)
  =a\left(\frac12-\frac{2}{\mu_0^2}\right)-1.5\cdot10^{-5}>0
\]
for $\Delta\ge\Delta_0$. At $\Delta=\Delta_0$ this function is greater than $1.3$. Hence
\[
  \mathbb{P}(S_F<2)
  \le \exp(-1.5\cdot10^{-5}\Delta).
\]
\end{proof}

\begin{lemma}\label{lem:Fi}
For every large clique $C_i$ and every $\Delta\ge\Delta_0$,
\[
  \mathbb{P}(F_i)\le e^{-3.5\cdot10^{-5}\Delta}+e^{-1.7\cdot10^{-5}\Delta}+e^{-2.1\cdot10^{-5}\Delta}+e^{-1.5\cdot10^{-5}\Delta}.
\]
\end{lemma}

\begin{proof}
Expose the colors in stages. First expose the colors of $C$. By (11), the probability that $G_1\cap G_2$ fails is at most
$
  e^{-3.5\cdot10^{-5}\Delta}+e^{-1.7\cdot10^{-5}\Delta}.
$
Condition on $G_1\cap G_2$. Choose $B$ as in Lemma~\ref{lem:clean-active} and expose the endpoint colors of triples in $B$. By Lemma~\ref{lem:clean-active}, the probability that there is no clean active subfamily $R\subseteq B$ of size $r=\lfloor3.5\cdot10^{-4}\Delta\rfloor$ is at most
$
  e^{-2.1\cdot10^{-5}\Delta}.
$
Condition on such a family $R$, and expose all remaining colors. By Lemma~\ref{lem:survival}, the conditional probability that fewer than two triples of $R$ survive is at most
$
  e^{-1.5\cdot10^{-5}\Delta}.
$

If two clean active triples survive, then $F_i$ does not occur. Indeed, for a surviving triple $t$, the center $c_t$ is uncolored because $t$ is internally duplicate. By (17), the pairs $(a_t,x_t)$ and $(b_t,y_t)$ are nonadjacent same-color pairs, and survival means neither pair is deleted by any conflict. Moreover, all four vertices $a_t,b_t,x_t,y_t$ lie in $N(c_t)$, and the two target colors are distinct. Hence $N(c_t)$ contains two distinct repeated colors after conflict deletion. The supplied triples are vertex-disjoint, so two surviving triples give two distinct uncolored vertices of $C$ with the desired property. Thus $F_i$ fails.

The union bound gives
\[
\begin{aligned}
  \mathbb{P}(F_i)
  &\le e^{-3.5\cdot10^{-5}\Delta}
      +e^{-1.7\cdot10^{-5}\Delta}
      +e^{-2.1\cdot10^{-5}\Delta}
      +e^{-1.5\cdot10^{-5}\Delta}.
\end{aligned}
\]
\end{proof}

\subsection{Completion of the proof}

\begin{proof}[Proof of Theorem~\ref{main}]
Suppose, for a contradiction, that Theorem~\ref{main} fails, and choose a
minimal counterexample $G$ with maximum degree $\Delta\ge\Delta_0$ and
$\omega(G)<\Delta$. By Brooks' theorem, this is the nontrivial case
$\omega(G)\le \Delta-1$ and $\chi(G)\ge \Delta$.

Let $\mathcal{E}$ be the family of all bad events $A_v$, $E_i$, and $F_i$
arising from the random partial coloring with $q=\Delta-1$. Set
\[
\begin{aligned}
  p^*:=\max\{&
  e^{-2.43\cdot10^{-5}\Delta}+e^{-2.2\cdot10^{-5}\Delta},\\
  &e^{-0.15\Delta},\\
  &e^{-3.5\cdot10^{-5}\Delta}
   +e^{-1.7\cdot10^{-5}\Delta}
   +e^{-2.1\cdot10^{-5}\Delta}
   +e^{-1.5\cdot10^{-5}\Delta}\}.
\end{aligned}
\]
By Lemmas~\ref{lem:Av}, \ref{lem:Ei}, and \ref{lem:Fi}, every event in
$\mathcal{E}$ has probability at most $p^*$.

The number of vertices within distance $4$ of any fixed vertex is at most
\[
  1+\Delta+\Delta(\Delta-1)+\Delta(\Delta-1)^2+\Delta(\Delta-1)^3
  <\Delta^4.
\]
Hence each event $A_v$ is mutually independent of all but at most $\Delta^4$
other events of the form $A_u$.

Each event $E_i$ depends only on the colors of vertices in $S_i$ and vertices
within distance one of $S_i$. Similarly, each event $F_i$ depends only on the
colors of vertices in $S_i$ and vertices within distance two of $S_i$.
Consequently, each event in $\mathcal E$ is mutually independent of all but at
most
$
  d:=\Delta^5+2\Delta^4
$
other events in $\mathcal E$.

For $\Delta\ge\Delta_0$, the third term in the definition of $p^*$ is the
largest one. Hence
\[
\begin{aligned}
  e p^*(d+1)
  &\le
  e(\Delta^5+2\Delta^4+1)
  \bigl(
     e^{-3.5\cdot10^{-5}\Delta}
     +e^{-1.7\cdot10^{-5}\Delta}
     +e^{-2.1\cdot10^{-5}\Delta}
     +e^{-1.5\cdot10^{-5}\Delta}
  \bigr).
\end{aligned}
\]
Define
\[
\begin{aligned}
  h(\Delta):={}&
  1+\log(\Delta^5+2\Delta^4+1)  \\
  &+\log\!\left(
     e^{-3.5\cdot10^{-5}\Delta}
     +e^{-1.7\cdot10^{-5}\Delta}
     +e^{-2.1\cdot10^{-5}\Delta}
     +e^{-1.5\cdot10^{-5}\Delta}
  \right).
\end{aligned}
\]
A direct calculation gives
\[
  h(5.3\times10^6)<-1.08.
\]
Moreover, for $\Delta\ge5.3\times10^6$,
\[
\begin{aligned}
  h'(\Delta)
  \le
  \frac{5}{\Delta}+\frac{8}{\Delta^2}-1.5\cdot10^{-5}
  <0.
\end{aligned}
\]
Thus \(h(\Delta)<0\) for every \(\Delta\ge\Delta_0\), and therefore
\[
  e p^*(d+1)<1.
\]
By the symmetric Lov\'asz Local Lemma, with positive probability no event in
$\mathcal{E}$ occurs.

If no $A_v$ occurs, then every leftover vertex $v\in L$ has at least two
distinct repeated colors in its neighborhood after conflict deletion. If no
$E_i$ occurs, then every near-clique vertex $u_i$ has at least two uncolored
neighbors in $C_i$. If no $F_i$ occurs, then every large clique $C_i$ contains
two uncolored vertices whose neighborhoods each contain two distinct repeated
colors. Thus all three hypotheses of Lemma~\ref{coloring} hold. The random
partial $(\Delta-1)$-coloring extends to a full $(\Delta-1)$-coloring of $G$,
contradicting the choice of $G$ as a counterexample. This proves
Theorem~\ref{main}.
\end{proof}
\section*{Acknowledgements}

This work was supported by the National Key R\&D Program of China
(No.~2022YFA1006400) and the National Natural Science Foundation of China
(No.~12571376).

\end{document}